\documentclass[10pt,reqno]{amsart}

\usepackage{amsmath,amssymb,amsthm,graphicx,enumerate,xcolor}

\newtheorem{thm}{Theorem}
\newtheorem{lem}[thm]{Lemma}

\theoremstyle{definition}
\newtheorem*{rem*}{Remark}

\begin{document}

\title[Maximal perimeters]{Maximal perimeters of polytope sections\\ and origin-symmetry}
\author{Matthew Stephen}

\address{Department of Mathematical and Statistical Sciences, University of Alberta, Edmonton, Alberta T6G 2G1, Canada} \email{mastephe@ualberta.ca}
\thanks{This research was supported in part by NSERC}

%52B15 - Symmetry properties of polytopes
%52A20 - Convex sets in $n$ dimensions
%52A38 - Length, area, volume

\begin{abstract}
Let $P\subset\mathbb{R}^n$ $(n\geq 3)$ be a convex polytope containing the origin in its interior. Let $\mbox{vol}_{n-2} \big( \mbox{relbd} ( P\cap\lbrace t\xi + \xi^\perp \rbrace ) \big)$ denote the $(n-2)$-dimensional volume of the relative boundary of $P\cap\lbrace t\xi + \xi^\perp \rbrace$ for $t\in\mathbb{R}$, $\xi\in S^{n-1}$. We prove the following: if
\begin{align*}
\mbox{vol}_{n-2} \Big( \mbox{relbd} \big( P\cap\xi^\perp \big) \Big)
	= \max_{t\in\mathbb{R}} \mbox{vol}_{n-2} \Big( \mbox{relbd} \big( P\cap\lbrace t\xi + \xi^\perp \rbrace \big) \Big) 
	\ \ \forall \ \ \xi\in S^{n-1},
\end{align*}
then $P$ is origin-symmetric, i.e. $P = -P$. Our result gives a partial affirmative answer to a conjecture by Makai, Martini, and \'Odor. We also characterize the origin-symmetry of $C^1$ convex bodies in terms of the dual quermassintegrals of their sections; this can be seen as a dual version of the conjecture of Makai et al. 
\end{abstract}

\maketitle

\section{Introduction}

A convex body $K\subset\mathbb{R}^n$ is a convex and compact subset of $\mathbb{R}^n$ with non-empty interior. Convex bodies are extensively studied objects in convex geometry; see the standard references \cite{Ga} and \cite{Sc} for known results and open problems. We say $K$ is origin-symmetric if it is equal to its reflection through the origin, i.e. $K = -K$. Many results for convex bodies depend on the presence of origin-symmetry. For example, the well-known Funk Section Theorem (e.g. Theorem 7.2.6 in \cite{Ga}) states that whenever $K_1,K_2\subset\mathbb{R}^n$ are origin-symmetric convex bodies such that 
\begin{align*}
\mbox{vol}_{n-1}\Big( K_1\cap\xi^\perp\Big) = \mbox{vol}_{n-1}\Big( K_2\cap\xi^\perp\Big) \quad \mbox{for all} \quad \xi\in S^{n-1} ,
\end{align*}
then necessarily $K_1=K_2$. Here, $\xi^\perp := \lbrace x\in\mathbb{R}^n :\, \langle x,\xi\rangle=0\rbrace$.

It is both interesting and useful to find properties which are equivalent to origin-symmetry. Several such characterizations are already known. For example, Falconer \cite{Fa} showed that a convex body $K\subset\mathbb{R}^n$ is origin-symmetric if and only if every hyperplane through the origin splits $K$ into two halves of equal $n$-dimensional volume. Brunn's Theorem implies that if $K$ is origin-symmetric, then
\begin{align}\label{MMO}
\mbox{vol}_{n-1} \Big( K\cap\xi^\perp\Big) 
	= \max_{t\in\mathbb{R}} \mbox{vol}_{n-1} \Big( K\cap\lbrace t\xi + \xi^\perp\rbrace \Big) \quad \mbox{for all} \quad \xi\in S^{n-1} .
\end{align}
Here, $\lbrace t\xi + \xi^\perp\rbrace$ denotes the translate of $\xi^\perp$ containing $t\xi$. Using a particular integral transform, Makai, Martini, and \'Odor \cite{MMO} proved the converse statement: a convex body $K$ which contains the origin in its interior and satisfies (\ref{MMO}) must be origin-symmetric. In fact, they proved a more general statement for star bodies. Ryabogin and Yaskin \cite{RY} gave an alternate proof using Fourier analysis, as well as a new characterization of origin-symmetry via conical sections. A stability version of the result of Makai et al. was established in \cite{SY}.

Makai et al. \cite{MMO} conjectured a further characterization of origin-symmetry in terms of the quermassintegrals of sections. Recall that the quermassintegrals $W_l(K)$ of a convex body $K\subset\mathbb{R}^n$ arise as coefficients in the expansion 
\begin{align*}
\mbox{vol}_n\Big( K + t B_2^n(o,1) \Big) = \sum_{l=0}^n \genfrac(){0pt}{0}{n}{l} W_l(K) \, t^l , \qquad t\geq 0. 
\end{align*}
The addition of sets here is the well-known Minkowski addition 
\begin{align*}
K + t B_2^n(o,1) := \big\lbrace x + t y : \, x\in K, \ y\in B_2^n(o,1) \big\rbrace .
\end{align*}
Refer to \cite{Sc} for a thorough overview of mixed volumes and quermassintegrals. For any $0\leq l\leq n-2$ and $\xi\in S^{n-1}$, consider the quermassintegral $W_l\big( (K - t\xi)\cap\xi^\perp \big)$ of the $(n-1)$-dimensional convex body $(K - t\xi)\cap\xi^\perp$ in $\xi^\perp$. If $K$ is origin-symmetric, then the monotonicity and positive multilinearity of mixed volumes together with the Alexandroff-Fenchel inequality imply 
\begin{align}\label{MMO conjecture}
W_l\Big( K\cap\xi^\perp \Big) = \max_{t\in\mathbb{R}} W_l\Big( (K - t\xi)\cap\xi^\perp \Big) \quad \mbox{for all} \quad \xi\in S^{n-1}. 
\end{align}
For $l=0$, (\ref{MMO conjecture}) is equivalent to (\ref{MMO}), as $W_0\big( (K - t\xi)\cap\xi^\perp \big)$ is the $(n-1)$-dimensional volume of $(K - t\xi)\cap\xi^\perp$. Makai et al. conjectured that if $K$ contains the origin in its interior and satisfies (\ref{MMO conjecture}) for any $1\leq l\leq n-2$, it must be origin-symmetric. Makai and Martini \cite{MM} proved a local variant of the conjecture for smooth perturbations of the Euclidean ball; otherwise, the problem is completely open.

In this paper, we consider the case of convex polytopes which satisfy (\ref{MMO conjecture}) for $l=1$. A convex polytope $P\subset\mathbb{R}^n$ is a convex body which is the convex hull of finitely many points. It is common practice to restrict unsolved problems for general convex bodies to the class of polytopes (e.g. \cite{MyR,Y1,Y2,YY,Z}), because polytopes have additional structure. Up to a constant depending on the dimension, $W_1\big( (P - t\xi)\cap\xi^\perp\big)$ is the $(n-2)$-dimensional surface area of the $(n-1)$-dimensional polytope $(P - t\xi)\cap\xi^\perp$ in $\xi^\perp$. Letting $\mbox{vol}_{n-2} \big( \mbox{relbd} ( P\cap\lbrace t\xi + \xi^\perp \rbrace ) \big)$ denote the $(n-2)$-dimensional volume of the relative boundary of $P\cap\lbrace t\xi + \xi^\perp \rbrace$, we prove the following:

\begin{thm}\label{Main SA}
Let $P\subset\mathbb{R}^n$ $(n\geq 3)$ be a convex polytope containing the origin in its interior, and such that
\begin{align}\label{Max_Perim_Property}
\mbox{vol}_{n-2} \Big( \mbox{relbd} \big( P\cap\xi^\perp \big) \Big)
	= \max_{t\in\mathbb{R}} \mbox{vol}_{n-2} \Big( \mbox{relbd} \big( P\cap\lbrace t\xi + \xi^\perp \rbrace \big) \Big) 
\end{align}
for all $\xi\in S^{n-1}$. Then $P = -P$. 
\end{thm}

We introduce some notation and simple lemmas in Section 2. The proof of Theorem \ref{Main SA} is presented in Section 3. Finally, in Section 4, we briefly explain how to characterize the origin-symmetry of $C^1$ convex bodies using the dual quermassintegrals of sections; this is a dual version of the conjecture of Martini et al. \cite{MMO}.

\section{Some Notation \& Auxiliary Lemmas}

The origin in $\mathbb{R}^n$ is denoted by $o$. The affine hull and linear span of a set $A\subset\mathbb{R}^n$ are respectively denoted by $\mbox{aff}(A)$ and $\mbox{span}(A)$. We let $\mathbb{R}x := \mbox{span}(x)$ be the line through $x\in\mathbb{R}^n\backslash\lbrace 0\rbrace$ and the origin. We use $|\,\cdot\, |_2$ for the Euclidean norm, $B_2^n(x,r)$ for the Euclidean ball of radius $r>0$ centred at $x\in\mathbb{R}^n$, and $S^{n-1}$ for the unit sphere. For $\xi\in S^{n-1}$, we define $\xi^+ := \lbrace x\in\mathbb{R}^n:\, \langle x,\xi\rangle\geq 0\rbrace$, $\xi^- := \lbrace x\in\mathbb{R}^n:\, \langle x,\xi\rangle\leq 0\rbrace$, and $S^{n-1}(\xi,\varepsilon) := S^{n-1}\cap B_2^n(\xi,\varepsilon)$ for small $\varepsilon > 0$. Finally, $[\xi_1,\xi_2] := S^{n-1} \cap \lbrace \alpha\xi_1 + \beta\xi_2 : \, \alpha,\beta \geq 0 \rbrace$ gives the geodesic connecting linearly independent $\xi_1,\xi_2\in S^{n-1}$.

For any $(n-2)$-dimensional polytope $G\subset\mathbb{R}^n$ which does not contain the origin, define $\eta_G\in S^{n-1}$ to be the unique unit vector for which 
\begin{itemize}
\item the line $\mathbb{R}\eta_G$ and $\mbox{aff}(G)$ intersect orthogonally; 
\item $G\subset\eta_G^-:= \lbrace x\in\mathbb{R}^n:\, \langle x,\eta_G\rangle\leq 0\rbrace$. 
\end{itemize}
For each $t>0$, 
\begin{align*}
\mbox{reflec}(G,t) := \big\lbrace x\in\mathbb{R}^n:\, \langle x,\eta_G\rangle = t \mbox{ and the line } \mathbb{R}x
	\mbox{ interesects } G \big\rbrace 
\end{align*}
is an $(n-2)$-dimensional polytope in $\mathbb{R}^n$; see Figure \ref{reflec}. In words, $\mbox{reflec}(G,t)$ is the homothetic copy of $-G$ lying in $\{ t\eta_G + \eta_G^\perp \}$, so that every line connecting a vertex of $\mbox{reflec}(G,t)$ to the corresponding vertex of $G$ passes through the origin. 
\begin{figure}
\includegraphics[width=0.75\linewidth]{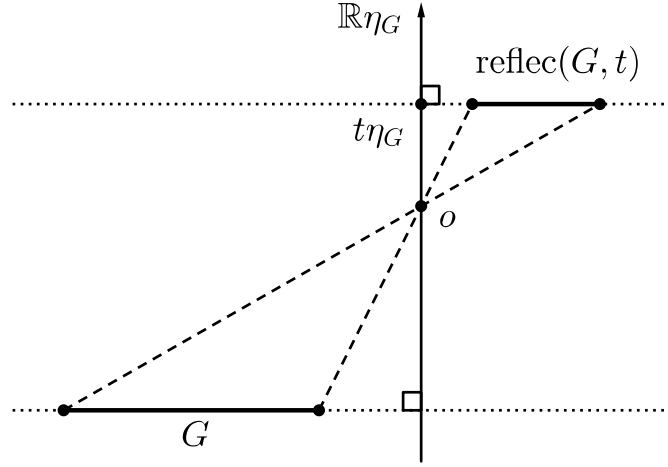}
\caption{The geometric meaning of $\mbox{reflec}(G,t)$.}
\label{reflec}
\end{figure}

\begin{lem}\label{avoiding vertices}
Let $Q\subset\mathbb{R}^n$ be a polytope for which the origin is not a vertex. Let $S^{n-1}(\theta_0,\varepsilon)$ be a spherical cap of radius $\varepsilon >0$ centred at $\theta_0\in S^{n-1}$. There exists $\theta\in S^{n-1}(\theta_0,\varepsilon)$ such that $\theta^\perp$ does not contain any vertices of $Q$. 
\end{lem}

\begin{proof}
If $u_1,\ldots, u_d$ are the vertices of $Q$, choose any $\theta$ from the non-empty set $S^{n-1}(\theta_0,\varepsilon) \backslash \left( u_1^\perp, \ldots, u_d^\perp \right)$. 
\end{proof}

The proof of the following lemma is trivial.

\begin{lem}\label{determinant derivatives}
Let $I\subset\mathbb{R}$ be an open interval. Let $\lbrace f_j\rbrace_{j=1}^n$ be a collection of differentiable $\mathbb{R}^n$-valued functions on $I$. Define $F(t) := \det \big(f_1(t), \ldots, f_n(t)\big)$. Then $F$ is differentiable on $I$ with
\begin{align*}
F'(t) = \sum_{j=1}^n \det\Big( f_1(t), \ldots, f_{j-1}(t), f_j'(t), f_{j+1}(t), \ldots, f_n(t) \Big) .
\end{align*}
\end{lem}

\section{Proof of Theorem \ref{Main SA}}

Let $P$ be a convex polytope containing the origin in its interior and satisfying (\ref{Max_Perim_Property}) for all $\xi\in S^{n-1}$. Our proof has two distinct parts.

We first need to prove that 
\begin{align}\label{1st Part}
\mbox{reflec}(G,t) \ \mbox{is an} \ (n-2)\mbox{-dimensional face of} \ P \ \mbox{for some} \ t>0
\end{align}
whenever $G$ is an $(n-2)$-dimensional face of $P$. To the contrary, we suppose $G_0$ is an $(n-2)$-dimensional face of $P$ for which (\ref{1st Part}) is false. We find a special spherical cap $S^{n-1}(\xi_0,\varepsilon)$. For every $\xi\in S^{n-1}(\xi_0,\varepsilon)$, $\xi^\perp$ misses all the vertices of $P$, while intersecting $G_0$ and no other $(n-2)$-dimensional faces which are parallel to $G_0$. We derive a ``nice" equation from (\ref{Max_Perim_Property}) which is valid for all $\xi\in S^{n-1}(\xi_0,\varepsilon)$. Forgetting the geometric meaning, we analytically extend this nice equation to all $\xi\in S^{n-1}$, excluding a finite number of great subspheres. Studying the behaviour near one of these subspheres, we arrive at our contradiction.

We conclude that for every vertex $v$ of $P$, the line $\mathbb{R}v$ contains another vertex $\widetilde{v}$ of $P$. In the second part of our proof, we prove that $\widetilde{v} = -v$. Hence, $P$ is origin-symmetric.

\subsection{First Part}

Assume there is an $(n-2)$-dimensional face $G_0$ of $P$ such that $\mbox{reflec}(G_0,t)$ is not an $(n-2)$-dimensional face of $P$ for any $t>0$. By the convexity of $P$, the intersection of $\mbox{aff}(o,G_0)$ with $P$ contains at most one other $(n-2)$-dimensional face of $P$ (besides $G_0$) which is parallel to $G_0$. If such a face exists, it must lie in $\lbrace t\eta_G + \eta_G^\perp \rbrace$ for some $t>0$ because $P$ contains the origin in its interior. We still allow that $\mbox{reflec}(G_0,t)$ may have a non-empty intersection with another $(n-2)$-dimensional face of $P$ for some $t>0$. However, we can additionally assume without loss of generality that $\mbox{reflec}(G_0,t)$ is not contained within an $(n-2)$-dimensional face of $P$ for any $t>0$.

\begin{lem}\label{plane_lemma}
There is a $\xi_0\in S^{n-1}$ such that
\begin{enumerate}[(i)]
\item the hyperplane $\xi_0^\perp$ does not contain any vertices of $P$;
\item $\xi_0^\perp$ intersects $G_0$ but no other $(n-2)$-dimensional faces of $P$ parallel to $G_0$;
\item there is exactly one vertex $v$ of $G_0$ contained in $\xi_0^+:= \lbrace x\in\mathbb{R}^n:\, \langle x,\xi_0\rangle\geq 0\rbrace$. 
\end{enumerate}
\end{lem}

\begin{proof}
Choose $\theta\in S^{n-1}$ so that $\theta^\perp = \mbox{aff}(o,G_0)$. Let $\eta := \eta_{G_0}$ be the unit vector defined as before.

There are two possibilities: either $\lbrace x\in\mathbb{R}^n:\, \langle x,\eta\rangle =t\rbrace \cap \theta^\perp$ does not contain an $(n-2)$-dimensional face of $P$ for any $t>0$, or it does for exactly one $t_0>0$. If the first case is true, we can of course choose an affine $(n-3)$-dimensional subspace $\widetilde{L}$ lying within $\mbox{aff}(G_0)$ which does not pass through any vertices of $G_0$, and separates exactly one vertex $v\in G_0$ from the others.

Suppose the second case is true, i.e. $H:=\lbrace x\in\mathbb{R}^n:\, \langle x,\eta\rangle =t_0\rbrace \cap \theta^\perp$ contains an $(n-2)$-dimensional face $G$ of $P$. By definition and assumption, $\mbox{reflec}(G_0,t_0)$ lies in $H$ and is not contained in $G$. We can choose $\widetilde{L}\subset H$ to be an $(n-3)$-dimensional affine subspace which, within $H$, strictly separates exactly one vertex $\widetilde{v}\in \mbox{reflec}(G_0,t_0)$ from both $G$ and the remaining vertices of $\mbox{reflec}(G_0,t_0)$. Let $v$ be the vertex of $G_0$ lying on the line $\mathbb{R}\widetilde{v}$.

Regardless of which case was true, set $L:=\mbox{aff}(o,\widetilde{L})\subset\theta^\perp$. The $(n-2)$-dimensional subspace $L$ intersects $G_0$ but no other $(n-2)$-dimensional faces of $P$ parallel to $G_0$, and separates $v$ from the remaining vertices of $G_0$. Perturbing $L$ if necessary (see Lemma \ref{avoiding vertices}), $L$ also does not intersect any vertices of $P$. Choose $\phi\in S^{n-1}\cap\theta^\perp\cap L^\perp$.

Define the slab $\theta^\perp_\alpha := \lbrace x\in\mathbb{R}^n:\, |\langle x,\theta\rangle|\leq\alpha\rbrace$, with $\alpha>0$ small enough so that $\theta^\perp_\alpha$ only contains vertices of $P$ lying in $\theta^\perp$. Necessarily, $\theta^\perp_\alpha$ also only contains the $(n-2)$-dimensional faces of $P$ parallel to $G_0$ which lie entirely in $\theta^\perp$. Choose $\beta>0$ large enough so that the slab $\phi^\perp_\beta$ contains $P$. Let $\xi_0\in S^{n-1}$ be such that $\xi_0^\perp = \mbox{aff}\big( o,\, \alpha\theta + \beta\phi + L\big)$ and $\langle v,\xi_0\rangle \geq 0$. We then have $\xi_0^\perp\cap P\subset \theta^\perp_\alpha$ and $\xi_0^\perp\cap\theta^\perp = L$. It follows from the construction of $\theta^\perp_\alpha$ and $L$ that $\xi_0$ has the desired properties. 
\end{proof}

Let $\lbrace E_i\rbrace_{i\in I}$ and $\lbrace F_j\rbrace_{j\in J}$ respectively be the edges and facets (i.e. $(n-1)$-dimensional faces) of $P$ intersecting $\xi_0^\perp$. Consider a spherical cap $S^{n-1}(\xi_0,\varepsilon)$ of radius $\varepsilon >0$ centred at $\xi_0$. For $\varepsilon >0$ small enough, the set
\begin{align*}
\big\lbrace x\in\mathbb{R}^n:\, |\langle x,\xi\rangle|\leq\varepsilon \mbox{ for some } \xi\in S(\xi_0,\varepsilon) \big\rbrace
\end{align*}
does not contain any vertices of $P$. Consequently, the map 
\begin{align*}
t \mapsto \mbox{vol}_{n-2} \Big( \mbox{relbd} \big( P\cap\lbrace t\xi + \xi^\perp\rbrace \big) \Big) 
	= \sum_{j\in J} \mbox{vol}_{n-2} \big( F_j\cap\lbrace t\xi + \xi^\perp\rbrace \big)
\end{align*}
is differentiable in a neighbourhood of $t=0$ for each $\xi\in S^{n-1}(\xi_0,\varepsilon)$. Therefore, (\ref{Max_Perim_Property}) implies
\begin{align}\label{equality}
\sum_{j\in J} \frac{d}{dt} \mbox{vol}_{n-2} \big( F_j\cap\lbrace t\xi + \xi^\perp\rbrace \big) \bigg|_{t=0} = 0
\end{align}
for every $\xi\in S^{n-1}(\xi_0,\varepsilon)$. We need to find an expression for this derivative.

For each $i\in I$, let $u_i + l_i s$ be the line in $\mathbb{R}^n$ containing $E_i$; $u_i$ is a point on the line, $l_i$ is a unit vector parallel to the line, and $s$ is the parameter. Clearly, $\lbrace t\xi+\xi^\perp\rbrace$ intersects the same edges and facets as $\xi_0^\perp$ for every $\xi\in S^{n-1}(\xi_0,\varepsilon)$ and $|t|\leq\varepsilon$. The intersection point of $\lbrace t\xi+\xi^\perp\rbrace$ with the edge $E_i$ is given by  
\begin{align*}
p_i(\xi,t) := u_i + l_i \left( \frac{ t-\langle u_i,\xi\rangle}{\langle l_i,\xi\rangle} \right) .
\end{align*}
Note that $\xi_0^\perp$ intersects exactly those edges of $G_0$ which are adjacent to the vertex $v$. Whenever $E_i$ is an edge of $G_0$ adjacent to $v$, we put $u_i := v$ and choose $l_i$ so that it gives the direction from another vertex of $G_0$ to $v$; this ensures $\langle l_i,\xi_0\rangle >0$.

For each $j\in J$, there is a pair of vertices from the facet $F_j$ such that the line through them does not lie in a translate of $\mbox{aff}(G_0)$, and with one of the vertices on either side of $\xi_0^\perp$. Translating this line if necessary, we obtain an auxiliary line $w_j+m_j s$ which 
\begin{itemize}
\item lies within $\mbox{aff}(F_j)$ and intersects the relative interior of $F_j$;
\item is transversal to $\xi^\perp$ for every $\xi\in S^{n-1}(\xi_0,\varepsilon)$;
\item does not lie within an $(n-2)$-dimensional affine subspace parallel to $\mbox{aff}(G_0)$. 
\end{itemize}
Again, $w_j$ is a point on the line, $m_j$ is a unit vector parallel to the line, and $s$ is the parameter. The intersection point of $\lbrace t\xi+\xi^\perp\rbrace$ with $w_j+m_j s$ is given by 
\begin{align*}
q_j(\xi,t) := w_j + m_j \left( \frac{ t-\langle w_j,\xi\rangle}{\langle m_j,\xi\rangle} \right) .
\end{align*}
Note that we necessarily have $\xi_0\not\perp m_j$ for all $j\in J$.

Consider a facet $F_j$, and an $(n-2)$-dimensional face $G$ of $P$ which intersects $\xi_0^\perp$ and is adjacent to $F_j$. Observe that $G\cap\lbrace t\xi + \xi^\perp\rbrace$ is an $(n-3)$-dimensional face of the $(n-2)$-dimensional polytope $F_j\cap\lbrace t\xi + \xi^\perp\rbrace$, for each $\xi\in S^{n-1}(\xi_0,\varepsilon)$ and $|t|\leq\varepsilon$. Express $G\cap\lbrace t\xi + \xi^\perp\rbrace$ as a disjoint union of $(n-3)$-dimensional simplices whose vertices correspond to the vertices of $G\cap\lbrace t\xi + \xi^\perp\rbrace$; that is, each simplex has vertices $p_{i_1}(\xi,t),\ldots, p_{i_{n-2}}(\xi,t)$ for some $i_1,\ldots, i_{n-2}\in I$. Triangulating every such $(n-3)$-dimensional face $G\cap\lbrace t\xi + \xi^\perp\rbrace$ in this way, we get a triangulation of $F_j\cap\lbrace t\xi + \xi^\perp\rbrace$ by taking the convex hull of the simplices in its relative boundary with $q_j(\xi,t)$.

\begin{rem*}
The description and orientation of a simplex $\Delta$ in the triangulation of $F_j\cap\lbrace t\xi + \xi^\perp\rbrace$ in terms of the ordered vertices $\lbrace p_{i_1}(\xi,t),\ldots, p_{i_{n-2}}(\xi,t), q_j(\xi,t)\rbrace$ is independent of $\xi\in S(\xi_0,\varepsilon)$ and $|t|\leq\varepsilon$. 
\end{rem*}

Setting $n_j\in S^{n-1}$ to be the outer unit normal to $F_j$, $\mbox{vol}_{n-2} \big( F_j\cap\lbrace t\xi + \xi^\perp\rbrace \big)$ is then a sum of terms of the form 
\begin{align}\label{simplex_volume}
\mbox{vol}_{n-2}(\Delta) 
	= \frac{ \det\Big( p_{i_1}(\xi,t) - q_j(\xi,t),\ldots, p_{i_{n-2}}(\xi,t) - q_j(\xi,t), n_j,\xi \Big) }
		{ (n-2)! \sqrt{1 - \langle n_j,\xi\rangle^2 } } ;
\end{align}
see page 14 in \cite{Ga}, for example, for the volume formula for a simplex. We assume the column vectors in the determinant are ordered so that the determinant is positive. Differentiating (\ref{simplex_volume}) at $t=0$ with the help of Lemma \ref{determinant derivatives} gives
\begin{align}\label{simplex_deriv}
\frac{1}{ (n-2)! \sqrt{1 - \langle n_j,\xi\rangle^2 } }
	\sum_{\gamma=1}^{n-2} \det\Big( X_{i_1}(\xi), \ldots, \widetilde{X}_{i_\gamma}(\xi), \ldots, X_{i_{n-2}}(\xi), n_j,\xi \Big) ,
\end{align} 
where 
\begin{align*}
X_{i_\gamma}(\xi) &:= p_{i_\gamma}(\xi,0) - q_j(\xi,0) \\
	&\ = u_{i_\gamma} - \left(\frac{\langle u_{i_\gamma}, \xi \rangle}{ \langle l_{i_\gamma}, \xi \rangle }\right) l_{i_\gamma} 
	- w_j + \left(\frac{\langle w_j, \xi \rangle}{ \langle m_j, \xi \rangle }\right) m_j , \\
\widetilde{X}_{i_\gamma}(\xi) &:= \frac{d}{dt} \Big( p_{i_\gamma}(\xi,t) - q_j(\xi,t) \Big) \bigg|_{t=0}
	= \frac{ l_{i_\gamma} }{ \langle l_{i_\gamma}, \xi \rangle } - \frac{ m_j }{ \langle m_j, \xi \rangle } .
\end{align*}

The left hand side of equation (\ref{equality}) is a sum of expressions having the form (\ref{simplex_deriv}). That is, (\ref{equality}) is equivalent to 
\begin{align}\label{equality2}
\sum_{\Delta} \left( \frac{ 
	\sum_{\gamma=1}^{n-2} \det\Big( X_{i_1}(\xi), \ldots, \widetilde{X}_{i_\gamma}(\xi), \ldots, X_{i_{n-2}}(\xi), n_j,\xi \Big) } 
	{ (n-2)! \sqrt{1 - \langle n_j,\xi\rangle^2 } } \right)
	= 0 ,
\end{align}
where the first summation is over all appropriately ordered indices $\lbrace i_1,\ldots,i_{n-2},j\rbrace$ corresponding to vertices of simplices $\Delta$ in our triangulation of $P\cap\xi_0^\perp$. Forget the geometric meaning of equation (\ref{equality2}). Clearing denominators on the left side of equation (\ref{equality2}) gives a function of $\xi$ which we denote by $\Phi(\xi)$. Because $\Phi$ is a sum of products of scalar products of $\xi$ and terms $\sqrt{1-\langle n_j,\xi\rangle^2}$, we are able to consider $\Phi$ as a function on all of $S^{n-1}$ such that $\Phi\equiv 0$ on $S^{n-1}(\xi_0,\varepsilon)$.

\begin{lem}\label{zero_lemma}
$\Phi(\xi)= 0$ for all $\xi\in S^{n-1}$. 
\end{lem}

\begin{proof}
Suppose $\zeta\in S^{n-1}$ is such that $\Phi(\zeta)\neq 0$. We have $\Phi(n_j) = 0$ for all $j\in J$, so $\zeta\neq n_j$. There is a $\zeta_1$ from the relative interior of $S^{n-1}(\xi_0,\varepsilon)$ which is not parallel to $\zeta$, and is such that the geodesic $[\zeta_1,\zeta]$ connecting $\zeta_1$ to $\zeta$ contains none of the $n_j$. Choose $\zeta_2\in S^{n-1}$ which is perpendicular to $\zeta_1$, lies in $\mbox{aff}\big( [\zeta_1,\zeta]\big)$, and is such that $\langle \zeta_2,\zeta\rangle > 0$. Let $\widetilde{\Phi}$ be the restriction of $\Phi$ to $\xi\in [\zeta_1,\zeta]$, and adopt polar coordinates $\xi = \zeta_1 \cos(\phi) + \zeta_2\sin(\phi)$. As a function of $\phi\in \big[0,\arccos(\langle\zeta_1,\zeta\rangle)\big]$, $\widetilde{\Phi}$ is a sum of products of $\cos(\phi),\, \sin(\phi)$, and $\sqrt{ 1 - \langle n_j, \zeta_1 \cos(\phi) + \zeta_2\sin(\phi)\rangle^2}$. The radicals in the expression for $\widetilde{\Phi}$ are never zero because $[\zeta_1,\zeta]$ misses all of the $n_j$, so $\widetilde{\Phi}$ is analytic. Consequently, $\widetilde{\Phi}$ must be identically zero, as it vanishes in a neighbourhood of $\phi = 0$. This is a contradiction.
\end{proof}

Lemma \ref{zero_lemma} implies that the equality in (\ref{equality2}) holds for all $\xi\in S^{n-1}\backslash A$, where $A$ is the union over $i\in I$ and $j\in J$ of the unit spheres in $l_i^\perp$ and $m_j^\perp$. Of course, we have $\pm n_j\in S^{n-1}\cap m_j^\perp$ for each $j\in J$, so $\lbrace \pm n_j\rbrace_{j\in J}\subset A$. We will consider the limit of the left side of equation (\ref{equality2}) along a certain path in $S^{n-1}\backslash A$ which terminates at a point in $A$.

The $(n-2)$-dimensional face $G_0$ is the intersection of two facets of $P$ belonging to $\lbrace F_j\rbrace_{j\in J}$, say $F_1$ and $F_2$. The normal space of $G_0$ is two dimensional and spanned by $n_1$ and $n_2$, the outer unit normals of $F_1$ and $F_2$. Consider the non-degenerate geodesic 
\begin{align*}
[n_1,n_2] := \left\lbrace \widetilde{n}_s := \frac{ (1-s) n_1 + s n_2 }{ \big| (1-s) n_1 + s n_2 \big|_2 } \,
	\Bigg| \, 0\leq s \leq 1 \right\rbrace 
\end{align*}
in $S^{n-1}\cap G_0^\perp \cong S^1$. This arc is not contained in the normal space of any other $(n-2)$-dimensional face of $P$ intersected by $\xi_0^\perp$, because $\xi_0^\perp$ does not intersect any other $(n-2)$-dimensional faces parallel to $G_0$; nor is $[n_1,n_2]$ contained in $m_j^\perp$ for any $j\in J$, because $m_j$ is not contained in a translate of $\mbox{aff}(G_0)$. Therefore, we can fix $0<s_0<1$ so that
\begin{itemize}
\item $\widetilde{n}:=\widetilde{n}_{s_0}$ is not a unit normal for any $(n-2)$-dimensional face of $P$ intersected by $\xi_0^\perp$, besides $G_0$;
\item $\widetilde{n}\not\perp m_j$, hence $\widetilde{n} \neq \pm n_j$, for all $j\in J$.
\end{itemize}
We additionally select $s_0$ so that it is not among the finitely many roots of the function 
\begin{align}\label{s0}
(0,1)\ni s \mapsto \frac{ -s }{ \sqrt{ 1 - \langle n_1, \widetilde{n}_s\rangle^2 } }
	+ \frac{ 1-s }{ \sqrt{ 1 - \langle n_2, \widetilde{n}_s\rangle^2 } } .
\end{align}
Observe that $\langle v, \widetilde{n} \rangle >0$ because $\langle v,n_1\rangle > 0$ and $\langle v,n_2\rangle > 0$.

For $\delta >0$, define the unit vector
\begin{align*}
\xi_\delta := \frac{ \widetilde{n} + \delta \xi_0 }{ | \widetilde{n} + \delta \xi_0 |_2 } .
\end{align*}
Clearly, 
\begin{align*}
\lim_{\delta\rightarrow 0^+} \xi_\delta = \widetilde{n}\in S^{n-1}\cap G_0^\perp 
	\subset \bigcup_{i\in I} S^{n-1}\cap l_i^\perp \subset A . 
\end{align*}
We have $\langle \xi_\delta, l_i\rangle , \langle \xi_\delta, m_j\rangle \neq 0$ for all $i\in I, j\in J$ whenever 
\begin{align*}
0 < \delta < \min \left\lbrace \frac{ |\langle \widetilde{n}, l_i\rangle| }{ |\langle \xi_0, l_i\rangle| } , \
	\frac{ |\langle \widetilde{n}, m_j\rangle| }{ | \langle \xi_0, m_j\rangle | } \ 
	\Bigg| \ i\in I \mbox{ such that } \langle \widetilde{n}, l_i\rangle \neq 0, \ j\in J \right\rbrace .
\end{align*}
The previous minimum is well-defined and positive, because $\xi_0\not\perp l_i,m_j$ and $\widetilde{n}\not\perp m_j$ for all $i\in I$, $j\in J$. So $\xi_\delta \in S^{n-1}\backslash A$ for small enough $\delta >0$.

Now, replace $\xi$ with $\xi_\delta$ in (\ref{equality2}), multiply both sides of the resulting equation by $\delta^{n-2}$, and take the limit as $\delta$ goes to zero. Consider what happens to the expressions (\ref{simplex_deriv}) multiplied by $\delta^{n-2}$ in this limit. We have 
\begin{align*}
\lim_{\delta\rightarrow 0^+} \delta X_{i_\gamma}(\xi_\delta) 
	&= \delta u_{i_\gamma} - \delta \left(\frac{\langle u_{i_\gamma}, \widetilde{n} + \delta \xi_0 \rangle}
	{ \langle l_{i_\gamma}, \widetilde{n} + \delta \xi_0 \rangle }\right) l_{i_\gamma} 
	- \delta w_j + \delta \left(\frac{\langle w_j, \widetilde{n} + \delta \xi_0 \rangle}
	{ \langle m_j, \widetilde{n} + \delta \xi_0 \rangle }\right) m_j \\
&= \begin{cases}
o \quad &\mbox{if} \quad \widetilde{n} \not\perp l_{i_\gamma} ; \\
- \frac{ \langle u_{i_\gamma}, \widetilde{n} \rangle }{ \left\langle l_{i_\gamma}, \xi_0 \right\rangle } l_{i_\gamma} 
	\quad &\mbox{if} \quad \widetilde{n} \perp l_{i_\gamma} ,
\end{cases}	
\end{align*}
and 
\begin{align*}
\lim_{\delta\rightarrow 0^+} \delta \widetilde{X}_{i_\gamma}(\xi_\delta) 
	&= \delta | \widetilde{n} + \delta \xi_0 |_2	
	\frac{ l_{i_\gamma} }{ \langle l_{i_\gamma}, \widetilde{n} + \delta \xi_0 \rangle } 
	- \delta | \widetilde{n} + \delta \xi_0 |_2
	\frac{ m_j }{ \langle m_j, \widetilde{n} + \delta \xi_0 \rangle } \\
&= \begin{cases}
o \quad &\mbox{if} \quad \widetilde{n} \not\perp l_{i_\gamma} ; \\
\frac{ l_{i_\gamma} }{ \left\langle l_{i_\gamma}, \xi_0 \right\rangle } \quad &\mbox{if} \quad \widetilde{n} \perp l_{i_\gamma} ,
\end{cases}
\end{align*}
because $\xi_0\not\perp l_i$ for all $i\in I$ and $\widetilde{n}\not\perp m_j$ for all $j\in J$. Therefore, expression (\ref{simplex_deriv}) vanishes in the limit if at least one index $i_\gamma$ in (\ref{simplex_deriv}) corresponds to an edge direction $l_{i_\gamma}$ which is not perpendicular to $\widetilde{n}$. If $l_{i_1}, \ldots, l_{i_{n-2}}$ are all perpendicular to $\widetilde{n}$, then expression (\ref{simplex_deriv}) becomes 
\begin{align}\label{simplex deriv limit}
\frac{ (-1)^{n-3} \det\Big( l_{i_1}, \ldots , l_{i_{n-2}}, n_j, \widetilde{n} \Big) }
	{ (n-2)! \sqrt{ 1 - \langle n_j,\widetilde{n}\rangle^2 } } 
	\sum_{\omega =1}^{n-2} \left( \langle u_{i_\omega}, \widetilde{n}\rangle^{-1}
	\prod_{\gamma = 1}^{n-2} \frac{ \langle u_{i_\gamma}, \widetilde{n} \rangle }
	{ \langle l_{i_\gamma}, \xi_0 \rangle } \right) .
\end{align}
If the determinant in (\ref{simplex deriv limit}) is non-zero, then $l_{i_1},\ldots, l_{i_{n-2}}$ are linearly independent. Therefore, $l_{i_1},\ldots, l_{i_{n-2}}$ span an $(n-2)$-dimensional plane which is parallel to the $(n-2)$-dimensional face $G$ of $P$ to which the edges $E_{i_1},\ldots, E_{i_{n-2}}$ belong. Necessarily, $\widetilde{n}$ will be a unit normal for $G$, so $G = G_0$ by our choice of $\widetilde{n}$. We conclude that the limit of (\ref{simplex_deriv}) only has a chance of being non-zero if (\ref{simplex_deriv}) corresponds to an $(n-2)$-dimensional simplex $\Delta_j$ in our triangulation of $F_j\cap\xi_0^\perp$, $j=1$ or $j=2$, with the base of $\Delta_j$ being an $(n-3)$-dimensional simplex in the triangulation of $G_0\cap\xi_0^\perp$.

If (\ref{simplex_deriv}) comes from such a $\Delta_1$ in the triangulation of $F_1\cap\xi_0^\perp$, then its limit is given by (\ref{simplex deriv limit}), and simplifies further to the non-zero term
\begin{align}\label{NonZero1}
\frac{ (-1)^{n-3} \langle v,\widetilde{n}\rangle^{n-3} s_0 \det\Big( l_{i_1}, \ldots , l_{i_{n-2}}, n_1, n_2 \Big) }
	{ (n-3)! \langle l_{i_1},\xi_0\rangle \times\cdots\times \langle l_{i_{n-2}},\xi_0\rangle 
	\big| (1-s_0)n_1 + s_0 n_2 \big|_2 \sqrt{ 1 - \langle n_1,\widetilde{n}\rangle^2 } } \neq 0 .
\end{align}
The distinct indices $i_1,\ldots , i_{n-2}$ correspond to the vertices of a simplex in the triangulation of $G_0\cap\xi_0^\perp$, ordered so that the expression in (\ref{simplex_volume}) for the facet $F_1$ is positive. The important fact that (\ref{NonZero1}) is non-zero is clear once we observe that the determinant is non-zero. Indeed, the unit vectors $l_{i_1},\ldots, l_{i_{n-2}}$ are necessarily linearly independent and perpendicular to both $n_1$ and $n_2$, because they give the directions for distinct edges of $G_0$ with the common vertex $v$. Similarly, when (\ref{simplex_deriv}) comes from such a $\Delta_2$ in the triangulation of $F_2\cap\xi_0^\perp$, it has the non-zero limit  
\begin{align}\label{NonZero2}
\frac{ (-1)^{n-3} \langle v,\widetilde{n}\rangle^{n-3} (1-s_0) \det\Big( l_{k_1}, \ldots , l_{k_{n-2}}, n_2, n_1 \Big) }
	{ (n-3)! \langle l_{k_1},\xi_0\rangle \times\cdots\times \langle l_{k_{n-2}},\xi_0\rangle 
	\big| (1-s_0)n_1 + s_0 n_2 \big|_2 \sqrt{ 1 - \langle n_2,\widetilde{n}\rangle^2 } } \neq 0.
\end{align}
The distinct indices $k_1,\ldots , k_{n-2}$ correspond to the vertices of a simplex in the triangulation of $G_0\cap\xi_0^\perp$, ordered so that the expression in (\ref{simplex_volume}) for the facet $F_2$ is positive.

We will now consider the signs of the determinants in (\ref{NonZero1}) and ({\ref{NonZero2}).

\begin{lem}\label{determinants}
The determinants $\det \big( l_{i_1}, \ldots, l_{i_{n-2}}, n_1, n_2 \big)$ in (\ref{NonZero1}) have the same sign for any collection of indices $i_1,\ldots , i_{n-2}$ with the previously described properties. The determinants in (\ref{NonZero2}) also all have the same sign. However, the signs of the determinants in (\ref{NonZero1}) and (\ref{NonZero2}) may differ. 
\end{lem}

\begin{proof}
Let $y\in G_0\cap\xi_0^\perp$. Consider the $(n-3)$ - dimensional subspace $L := \mbox{span}(G_0\cap\xi_0^\perp - y)$, which is orthogonal to $\mbox{span}( n_1, n_2, \xi_0)$, and the $(n-2)$ - dimensional subspace $\widetilde{L} = \mbox{span}( n_1, L)$. The projections $n_2|n_1^\perp$ and $\xi_0|n_1^\perp$ are non-zero and orthogonal to $\widetilde{L}$. Let $T:\mathbb{R}^n\rightarrow\mathbb{R}^n$ be the special orthogonal matrix which leaves $\widetilde{L}$ fixed, and rotates $n_2|n_1^\perp$ through the two - dimensional plane $\mbox{span}( n_2|n_1^\perp, \xi_0|n_1^\perp)$ to a vector parallel to, and with the same direction as, $\xi_0|n_1^\perp$. We have $\big( n_2|n_1^\perp \big) \perp (v-y)$, because $n_1, n_2 \perp (v-y)$. Since orthogonal transformations preserve inner products, 
\begin{align*}
\big\langle \xi_0|n_1^\perp, T(v-y)\big\rangle 
	&= \big| \xi_0|n_1^\perp \big|_2 \big| T(n_2|n_1^\perp) \big|_2^{-1} \big\langle T( n_2|n_1^\perp), T(v-y) \big\rangle \\
&= \big| \xi_0|n_1^\perp \big|_2 \big| T(n_2|n_1^\perp) \big|_2^{-1} \big\langle n_2|n_1^\perp, v-y \big\rangle \\
&= 0,
\end{align*}
and 
\begin{align*}
\big\langle n_1, T(v-y)\big\rangle = \big\langle T(n_1), T(v-y) \big\rangle = \langle n_1, v-y\rangle = 0 .
\end{align*}
We have $(\xi_0|n_1^\perp)^\perp \cap n_1^\perp = \mbox{span}\big( F_1\cap\xi_0^\perp - y \big)$, because $n_1,\xi_0 \perp \mbox{span}\big( F_1\cap\xi_0^\perp -y \big)$. Therefore, $T$ maps $v-y$ into $\mbox{span}\big( F_1\cap \xi_0^\perp - y \big)$, which also contains $q_1(\xi_0,0)-y$.

The subspace $L$ splits $\mbox{span}\big( F_1\cap \xi_0^\perp - y \big)$ into two halves. If $T(v-y)$ and $q_1(\xi_0,0)-y$ lie in the same half, let $\widetilde{T}:\mathbb{R}^n\rightarrow \mathbb{R}^n$ be the identity. If $T(v-y)$ and $q_1(\xi_0,0)-y$ lie in opposite halves, let $\widetilde{T}$ be the orthogonal transformation which leaves $L$ and $\mbox{span}\big( F_1\cap \xi_0^\perp - y \big)^\perp$ fixed, and reflects $T(v-y)$ across $L$. In either case, set $u := \widetilde{T}T(v-y)+y \in\mbox{aff}(F_1\cap\xi_0^\perp)$. We have that $u$ and $q_1(\xi_0,0)$ lie on the same side of $\mbox{aff}(G_0\cap\xi_0)$ in $\mbox{aff}(F_1\cap\xi_0^\perp)$. Also, $\widetilde{T}n_1 = n_1$ and $\widetilde{T}(\xi_0|n_1^\perp) = \xi_0|n_1^\perp$.

For any indices $i_1,\ldots, i_{n-2}$ from (\ref{NonZero1}), we find that
\begin{align}\label{determinant sign1}
&\det \Big( l_{i_1}, \ldots, l_{i_{n-2}}, n_1, n_2 \Big) \nonumber \\
&= \det \left( \frac{ v - p_{i_1}(\xi_0,0) }{ |v - p_{i_1}(\xi_0,0)|_2 }, 
	\ldots, \frac{ v - p_{i_{n-2}} }{ |v - p_{i_{n-2}}(\xi_0,0)|_2 }, n_1, n_2 - \langle n_1,n_2\rangle n_1 \right) \nonumber \\
&= C \det \Big( \widetilde{T}T(v-y) - \widetilde{T}T(p_{i_1}(\xi_0,0)-y), \ldots \nonumber \\
&\hspace{3.5cm} \ldots, \widetilde{T}T(v-y) - \widetilde{T}T(p_{i_{n-2}}(\xi_0,0)-y), 
	\widetilde{T}Tn_1, \widetilde{T}T(n_2|n_1^\perp) \Big) \nonumber \\
&= \frac{ C (-1)^{n-2} \big| \xi_0|n_1^\perp \big|_2 }{ \big| T(n_2|n_1^\perp) \big|_2 } 
	\det \Big( p_{i_1}(\xi_0,0) - u, \ldots, p_{i_{n-2}}(\xi_0,0) - u, n_1, \xi_0|n_1^\perp \Big) \nonumber \\
&= \frac{ C (-1)^{n-2} \big| \xi_0|n_1^\perp \big|_2 }{ \big| T(n_2|n_1^\perp) \big|_2 } 
	\det \Big( p_{i_1}(\xi_0,0) - u, \ldots, p_{i_{n-2}}(\xi_0,0) - u, n_1, \xi_0 \Big) ,
\end{align}
where 
\begin{align*}
C = \pm \left( \prod_{\gamma = 1}^{n-2} |v - p_{i_\gamma}(\xi_0,0)|_2 \right)^{-1} .
\end{align*}
The sign of $C$ depends on the definition of $\widetilde{T}$. Importantly, the sign of 
\begin{align*}
\frac{ C (-1)^{n-2} \big| \xi_0|n_1^\perp \big|_2 }{ \big| T(n_2|n_1^\perp) \big|_2 }  
\end{align*}
is independent of any particular choice of appropriate indices in (\ref{NonZero1}). The function  
\begin{align*}
t \mapsto &\det \Big( p_{i_1}(\xi_0,0) - \big( (1-t)u + t q_1(\xi_0,0) \big), \ldots \\
&\hspace{3cm} \ldots, p_{i_{n-2}}(\xi_0,0) - \big( (1-t)u + t q_1(\xi_0,0) \big), n_1, \xi_0 \Big)
\end{align*}
is continuous for $t\in [0,1]$; it is also non-vanishing for such $t$ because the line segment connecting $u$ to $q_1(\xi_0,0)$ lies in $F_1\cap\xi_0^\perp$ and does not intersect $G_0\cap\xi_0^\perp$. By the Intermediate Value Theorem, the determinant in (\ref{determinant sign1}) must have the same sign as 
\begin{align*}%\label{determinant sign2}
\det \Big( p_{i_1}(\xi_0,0) - q_1(\xi_0,0), \ldots, p_{i_{n-2}}(\xi_0,0) - q_1(\xi_0,0), n_1, \xi_0 \Big) .
\end{align*}
Recalling formula (\ref{simplex_volume}), we recognize that the previous determinant is positive. We conclude that the sign of $\det \big( l_{i_1}, \ldots, l_{i_{n-2}}, n_1, n_2 \big)$ is independent of the choice of appropriate indices in (\ref{NonZero1}).

A similar argument shows that the sign of the determinant in (\ref{NonZero2}) is also independent of the choice of appropriate indices $k_1, \ldots, k_{n-2}$. 
\end{proof}

In view of Lemma \ref{determinants} and the expressions (\ref{NonZero1}) and (\ref{NonZero2}), we see that 
\begin{align}\label{limit}
&\lim_{\delta\rightarrow 0^+} \delta^{n-2} \sum_{\Delta} \left( \frac{ 
	\sum_{\gamma=1}^{n-2} \det\Big( X_{i_1}(\xi_\delta), \ldots, \widetilde{X}_{i_\gamma}(\xi_\delta), \ldots, X_{i_{n-2}}(\xi_\delta), 
	n_j,\xi_\delta \Big) }
	{ (n-2)! \sqrt{1 - \langle n_j,\xi_\delta\rangle^2 } } \right) 
	\nonumber \\ 
&= \sum \Bigg( \frac{ (-1)^{n-3} \langle v,\widetilde{n}\rangle^{n-3} s_0 \det\Big( l_{i_1}, \ldots , l_{i_{n-2}}, n_1, n_2 \Big) }
	{ (n-3)! \langle l_{i_1},\xi_0\rangle \times\cdots\times \langle l_{i_{n-2}},\xi_0\rangle 
	\big| (1-s_0)n_1 + s_0 n_2 \big|_2 \sqrt{ 1 - \langle n_1,\widetilde{n}\rangle^2 } } \nonumber \\
&\hspace{1.25cm} + \frac{ (-1)^{n-3} \langle v,\widetilde{n}\rangle^{n-3} (1-s_0) \det\Big( l_{k_1}, \ldots , l_{k_{n-2}}, n_2, n_1 \Big) }
	{ (n-3)! \langle l_{k_1},\xi_0\rangle \times\cdots\times \langle l_{k_{n-2}},\xi_0\rangle 
	\big| (1-s_0)n_1 + s_0 n_2 \big|_2 \sqrt{ 1 - \langle n_2,\widetilde{n}\rangle^2 } } \Bigg) \nonumber \\
&= \frac{ (-1)^{n-3} \langle v,\widetilde{n}\rangle^{n-3} }{ (n-3)! \big| (1-s_0)n_1 + s_0 n_2 \big|_2 } 
	\left( \frac{s_0}{ \sqrt{ 1 - \langle n_1,\widetilde{n}\rangle^2 } } 
	\pm \frac{ 1-s_0  }{ \sqrt{ 1 - \langle n_2,\widetilde{n}\rangle^2 } } \right) \\
&\hspace{5cm} \cdot \sum \frac{ \det\Big( l_{i_1}, \ldots , l_{i_{n-2}}, n_1, n_2 \Big) }
	{ \langle l_{i_1},\xi_0\rangle \times\cdots\times \langle l_{i_{n-2}},\xi_0\rangle } . \nonumber
\end{align}
The third and fourth summations are taken over indices $i_1,\ldots, i_{n-2}\in I$ corresponding to the vertices $\lbrace p_{i_1}(\xi_0,0), \ldots, p_{i_{n-2}}(\xi_0,0)\rbrace $ of simplices in the triangulation of $G_0\cap\xi_0^\perp$, ordered so that the expression in (\ref{simplex_volume}) is positive for $F_1$. For each set of indices $i_1,\ldots, i_{n-2}$, $k_1,\ldots, k_{n-2}$ is a suitable rearrangement so that (\ref{simplex_volume}) is positive for $F_2$. The $\pm$ in (\ref{limit}) depends on whether or not the determinants 
\begin{align*}
\det\Big( l_{i_1}, \ldots , l_{i_{n-2}}, n_1, n_2 \Big)
\end{align*}
have the same sign as the determinants 
\begin{align*}
\det\Big( l_{k_1}, \ldots , l_{k_{n-2}}, n_2, n_1 \Big) . 
\end{align*}
We see that (\ref{limit}) is non-zero because $\langle v,\widetilde{n} \rangle >0$, $s_0$ is not a root of (\ref{s0}), $\langle l_i,\xi_0\rangle > 0$ for all edges $E_i$ of $G_0$ intersected by $\xi_0^\perp$, and by Lemma \ref{determinants}. The limit being non-zero contradicts the equality in (\ref{equality2}).

\subsection{Second Part}

Therefore, for every $(n-2)$-dimensional face $G$ of $P$, $\mbox{reflec}(G,t)$ is also an $(n-2)$-dimensional face of $P$ for some $t>0$. From this fact, we can immediately conclude the following:
\begin{itemize}
\item If $v$ is a vertex of $P$, then the line $\mathbb{R}v$ contains exactly one other vertex of $P$. This second vertex, which we will denote by $\widetilde{v}$, necessarily lies on the opposite side of the origin as $v$.
\item If $u$ and $v$ are vertices of $P$ connected by an edge $E(u,v)$, then $\widetilde{u}$ and $\widetilde{v}$ are connected by an edge $E(\widetilde{u},\widetilde{v})$ parallel to $E(u,v)$.
\end{itemize}
We prove $P=-P$ by showing $\widetilde{v} = -v$ for every vertex $v$.

To the contrary, suppose there is a vertex $v$ for which $|v|_2<|\widetilde{v}|_2$. Let $\lbrace v_i\rbrace_{i=0}^k$ be a sequence of vertices of $P$ such that $v_0 = v$, $v_k = \widetilde{v}$, and the vertices $v_i$ and $v_{i+1}$ are connected by an edge $E(v_i,v_{i+1})$ for each $0\leq i\leq k-1$. It follows from the previous itemized observations that the triangle $T(o,v,v_1)$ with vertices $\lbrace o,\, v,\, v_1 \rbrace$ is similar to the triangle $T(o,\widetilde{v},\widetilde{v}_1)$; see Figure \ref{Triangles}. Given that $|v|_2<|\widetilde{v}|_2$, we must also have $|v_1|_2 < |\widetilde{v}_1|_2$. Continuing this argument recursively, the similarity of the triangle $T(o,v_i,v_{i+1})$ to the triangle $T(o,\widetilde{v}_i,\widetilde{v}_{i+1})$ implies $|v_{i+1}|_2 < |\widetilde{v}_{i+1}|_2$ for $1\leq i\leq k-1$. But then $|\widetilde{v}|_2 = |v_k|_2 < |\widetilde{v}_k|_2 = |v|_2$, which is a contradiction. 
\begin{figure}
\includegraphics[width=0.75\linewidth]{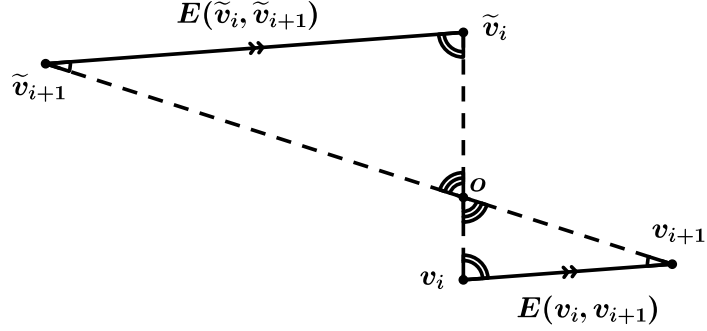}
\caption{ The triangle $T(o,v_i,v_{i+1})$ is similar to the triangle $T(o,\widetilde{v}_i,\widetilde{v}_{i+1})$ for $0\leq i\leq k-1$. }
\label{Triangles}
\end{figure}

\section{Dual Quermassintegrals of Sections}

Throughout this section, let $K\subset\mathbb{R}^n$ be a convex body containing the origin in its interior. We consider the radial sum
\begin{align*}
K\, \widetilde{+}\, t B_2^n(o,1) := \{o\} \cup \big\lbrace x\in\mathbb{R}^n\backslash\{o\} : |x|_2 \leq \rho_K(x / |x|_2) + t \big\rbrace 
	, \qquad t\geq 0,
\end{align*}
where
\begin{align*}
\rho_K(\xi) := \max \lbrace a>0 : a\xi \in K \rbrace, \qquad \xi\in S^{n-1}, 
\end{align*}
is the radial function for $K$. The so-called dual quermassintegrals $\widetilde{W}_l(K)$ arise as coefficients in the expansion 
\begin{align*}
\mbox{vol}_n\Big( K\, \widetilde{+}\, t B_2^n(o,1) \Big) = \sum_{l=0}^n \genfrac(){0pt}{0}{n}{l} \widetilde{W}_l(K) \, t^l , \qquad t\geq 0.  
\end{align*}
Dual quermassintegrals (and, more generally, dual mixed volumes) were introduced by Lutwak \cite{L}. See \cite{Ga, Sc} for further details. There are many parallels between quermassintegrals and dual quermassintegrals, so it is natural to consider the conjecture of Makai et al. \cite{MMO} in the dual setting. We pose and solve such a question.

For each integer $0\leq l\leq n-2$ and $\xi\in S^{n-1}$, we define the function 
\begin{align*}
\widetilde{W}_{l,\xi}(t) := \widetilde{W}_l\Big( (K - t\xi)\cap\xi^\perp \Big) , \qquad t\in \mathbb{R}, 
\end{align*}
where $\widetilde{W}_l\big( (K - t\xi)\cap\xi^\perp \big)$ is the dual quermassintegral of the $(n-1)$-dimensional convex body $(K - t\xi)\cap\xi^\perp$ in $\xi^\perp$. It follows from the dual Kubota formula (e.g. Theorem A.7.2 in \cite{Ga}) and Brunn's Theorem that 
\begin{align}\label{dual MMO conjecture}
\widetilde{W}_l(K\cap\xi^\perp) = \widetilde{W}_{l,\xi} (0)
	= \max_{t\in\mathbb{R}} \widetilde{W}_{l,\xi}(t) \quad \mbox{for all} \quad \xi\in S^{n-1} 
\end{align}
whenever $K$ is origin-symmetric. For $l=0$, (\ref{dual MMO conjecture}) is equivalent to (\ref{MMO}). We prove the converse statement when $K$ is a $C^1$ convex body; that is, the boundary of $K$ is a $C^1$ manifold, or equivalently $\rho_K\in C^1(S^{n-1})$.

\begin{thm}\label{Main Dual}
Suppose $K\subset\mathbb{R}^n$ is a $C^1$ convex body containing the origin in its interior. If $K$ satisfies (\ref{dual MMO conjecture}) for some $1\leq l\leq n-2$, then necessarily $K=-K$. 
\end{thm}

The proof of Theorem \ref{Main Dual} follows from formulas derived in \cite{Y0}. These formulas involve spherical harmonics, and the fractional derivatives of $\widetilde{W}_{l,\xi}$ at $t=0$. Recall that a spherical harmonic of dimension $n$ and degree $m$ is the restriction to $S^{n-1}$ of a real-valued harmonic and homogeneous polynomial of degree $m$ in $n$ variables. Importantly, there is an orthogonal basis of $L^2(S^{n-1})$ consisting of spherical harmonics. Furthermore, any two spherical harmonics of the same dimension and different degrees are orthogonal. The standard reference for spherical harmonics in convex geometry is \cite{Gr}.

Let $h$ be an integrable function on $\mathbb{R}$ which is $m$ times continuously differentiable in a neighbourhood of zero. Let $q\in\mathbb{C}\backslash \{ 0,1,\ldots, m\}$ with real part $-1< \mbox{Re}(q) < m+1$. The fractional derivative of $h$ of order $q$ at zero is given by 
\begin{align*}
h^{(q)}(0) =& \frac{1}{\Gamma(-q)} 
	\int_0^1 t^{-1-q} \left( h(t) - \sum_{k=0}^m \frac{d^k}{ds^k} h(s)\Big|_{s=0} \frac{t^k}{k!} \right) \, dt \\
&+ \frac{1}{\Gamma(-q)} \int_1^\infty t^{-1-q} h(t) \, dt 
	+ \frac{1}{\Gamma(-q)} \sum_{k=0}^m \frac{1}{k!(k-q)} \frac{d^k}{dt^k} h(t)\Big|_{t=0} ;
\end{align*}
see, for example, \cite{K}. Defining $h^{(k)}(0)$ by the limit for $k=0,1,\ldots,m$, we get an analytic function $q\mapsto h^{(q)}(0)$ for $q\in\mathbb{C}$ with $-1< \mbox{Re}(q) < m+1$. Fractional derivatives extend the notion of classical differentiation:
\begin{align*}
h^{(k)}(0) = (-1)^k \frac{d^k}{dt^k} h(t)\Big|_{t=0} \quad \mbox{for} \quad k\in\mathbb{N}.
\end{align*}
Note that $\widetilde{W}_{l,\xi}$ is continuously differentiable in a neighbourhood of zero when $K$ is $C^1$, so we can consider the fractional derivatives of $\widetilde{W}_{l,\xi}$ at zero of order $q$, $-1<\mbox{Re}(q) < 2$.

\begin{proof}[Proof of Theorem \ref{Main Dual}]
It is proven in \cite{Y0} (at the bottom of page 8, in their Theorem 2) that 
\begin{align}\label{Vlad's formula}
\int_{S^{n-1}} H_m^n(\xi)\, \widetilde{W}_{l,\xi}^{(q)}(0) \, d\xi 
	= \frac{ (n-1-l) \lambda_m(q) }{ ( n-1-q-l )(n-1) } \int_{S^{n-1}} H_m^n(\xi)\, \rho_K^{ n-1-q-l }(\xi)\, d\xi 
\end{align}
for all $-1<q<0$ and spherical harmonics $H_m^n$ of dimension $n$ and odd degree $m$. The multipliers $\lambda_m(q)$ in (\ref{Vlad's formula}) come from an application of the Funke-Hecke Theorem. Let $P_m^n$ denote the Legendre polynomial of dimension $n$ and odd degree $m$. It is shown in \cite{Y0} (on page 7) that, explicitly, $\lambda_m(q)$ is the fractional derivative of $f(t)=P_m^n(t)(1-t^2)^{(n-2-l)/2}$ of order $q$ at $t=0$. Therefore, $q\mapsto\lambda_m(q)$ is analytic, and (\ref{Vlad's formula}) can be extended to $-1 < q < 2$, $q\neq n-1-l$.

Observe that for odd integers $m$, 
\begin{align*}
\lim_{q\rightarrow 1} \int_{S^{n-1}} H_m^n(\xi)\, \rho_K^{ n-1-q-l }(\xi)\, d\xi 
	= \int_{S^{n-1}} H_m^n(\xi)\, d\xi = 0
\end{align*}
because of the orthogonality of spherical harmonics with different degrees. Taking the limit as $q$ approaches $1$ in (\ref{Vlad's formula}), we get 
\begin{align}\label{Vlad's formula limit}
&- \int_{S^{n-1}} H_m^n(\xi) \widetilde{W}_{l,\xi}'(0) \, dt = \int_{S^{n-1}} H_m^n(\xi) \widetilde{W}_{l,\xi}^{(1)}(0) \, dt \nonumber \\
&= \begin{cases}
\frac{ (n-1-l) \lambda_m(1) }{ (n-2-l) (n-1) } \int_{S^{n-1}} H_m^n(\xi) \rho_K^{n-2-l}(\xi) \, d\xi 
	\quad &\mbox{if} \quad l\neq n-2; \\
\frac{ \lambda_m(1) }{ n-1 } \int_{S^{n-1}} H_m^n(\xi) \log\big( \rho_K(\xi) \big) \, d\xi \quad &\mbox{if} \quad l=n-2,
\end{cases}
\end{align}
for all odd integers $m$. We use L'Hospital's rule to evaluate the limit for the case $l=n-2$.

Calculating 
\begin{align*}
\lambda_m(1) = f^{(1)}(0) = - \frac{d}{dt} P_m^n(t) (1-t^2)^{(n-2-l)/2} \Big|_{t=0} = - \frac{d}{dt} P_m^n(t) \Big|_{t=0} ,
\end{align*}
it then follows from Lemma 3.3.9 and Lemma 3.3.8 in \cite{Gr} that $\lambda_m(1)\neq 0$ for odd $m$. As (\ref{dual MMO conjecture}) implies $\widetilde{W}_{l,\xi}'(0) = 0$ for all $\xi\in S^{n-1}$, we conclude from (\ref{Vlad's formula limit}) and $\lambda_m(1)\neq 0$ that 
\begin{itemize}
\item if $l\neq n-2$, the spherical harmonic expansion of $\rho_K^{n-2-l}$ does not have any harmonics of odd degree;
\item if $l=n-2$, the spherical harmonic expansion of $\log(\rho_K)$ does not have any harmonics of odd degree. 
\end{itemize}
Consequently, $\rho_K$ must be an even function, so $K$ is origin-symmetric. 
\end{proof}

\begin{rem*}
It can be seen that Theorem \ref{Main Dual} is actually true for $C^1$ star bodies, i.e. compact sets with positive and $C^1$ radial functions. However, it is not necessary for origin-symmetric star bodies to satisfy (\ref{dual MMO conjecture}). 
\end{rem*}


\begin{thebibliography}{00}


\bibitem{Fa}
{\sc K.~J.~Falconer}, {\em Applications of a result on spherical integration to the theory of convex sets}, Amer. Math. Monthly 90 (1983), 690--693.


\bibitem{Ga}
{\sc R.~Gardner}, {\em Geometric Tomography}, second edition, Cambridge University Press, New York, 2006. 


\bibitem{Gr}
{\sc H.~Groemer}, {\em Geometric Applications of Fourier series and Spherical Harmonics}, Cambridge University Press, Cambridge, 1996. 


\bibitem{K}
{\sc A.~Koldobsky}, {\em Fourier Analysis in Convex Geometry}, American Mathematical Society, Providence RI, 2005. 


\bibitem{L}
{\sc E.~Lutwak}, {\em Dual mixed volumes}, Pacific J. Math. {\bf 58} (1975), 531--538.


\bibitem{MM}
{\sc E.~Makai and H.~Martini}, {\em Centrally symmetric convex bodies and sections having maximal quermassintegrals}, Studia Sci. Math. Hungar. {\bf 49} (2012), 189--199.


\bibitem{MMO}
{\sc E.~Makai, H.~Martini, T.~\'Odor}, {\em Maximal sections and centrally symmetric bodies}, Mathematika {\bf 47} (2000), 19--30.


\bibitem{MyR}
{\sc S.~Myroshnychenko and D.~Ryabogin}, {\em On polytopes with congruent projections or sections}, Adv. Math. {\bf 325} (2018), 482--504. 


\bibitem{RY}
{\sc D.~ Ryabogin and V.~Yaskin}, {\em Detecting symmetry in star bodies}, J. Math. Anal. Appl. {\bf 395} (2012), {509--514}.


\bibitem{Sc}
{\sc R.~Schneider}, {\em Convex Bodies: The Brunn-Minkowski Theory}, Cambridge University Press, Cambridge, 2014.


\bibitem{SY}
{\sc M.~Stephen and V.~Yaskin}, {\em Stability results for sections of convex bodies}, Trans. Amer. Math. Soc. {\bf 369} (2017), 6239--6261.


\bibitem{Y0}
{\sc V.~Yaskin}, {\em An extension of polynomial integrability to dual quermassintegrals}, available at arXiv:1803.00199 [math.MG].


\bibitem{Y1}
{\sc V.~Yaskin}, {\em On perimeters of sections of convex polytopes}, J. Math. Anal. Appl. {\bf 371} (2010), 447--453. 


\bibitem{Y2}
{\sc V.~Yaskin}, {\em Unique determination of convex polytopes by non-central sections}, Math. Ann. {\bf 349} (2011), 647--655. 


\bibitem{YY}
{\sc V.~Yaskin and M.~Yaskina}, {\em Thick sections of convex bodies}, Adv. in Appl. Math. {\bf 71} (2015), 174--189. 


\bibitem{Z}
{\sc G.~Zhu}, {\em The logarithmic Minkowski problem for polytopes}, Adv. Math. {\bf 262} (2014), 909-931. 


\end{thebibliography}
\end{document}